\documentclass[11pt]{article}
\usepackage{amsfonts,amsmath,amssymb,amsthm, amscd}
\usepackage{graphicx}

\numberwithin{equation}{section}
\newtheorem{theorem}{Theorem}[section]

\newtheorem{lemma}[theorem]{Lemma}
\newtheorem{cor}[theorem]{Corollary}

\theoremstyle{definition}

\def\<{{\langle}}
\def\>{{\rangle}}

\def\a{{\alpha}}

\def\Z{\mathbb Z}

\def\R{\mathbb R}
\def\T{\mathbb T}
\def\S{{\mathbb S}}

\def\a{\alpha}

\def\t{\tau}

\def\ni{\noindent} 
\begin{document}

\title{Three dimensions of knot coloring}

\author{J. Scott Carter\and Daniel S. Silver 
\and Susan G. Williams}

\maketitle %{\setlength{\linewidth}{2in}

%%%%%%%%%%%%%%%%%%%%%%%%%%%%%% 

\section{Introduction} \label{Intro}

\quad\quad  {\it Color is my day-long obsession, joy and torment.} --
Claude Monet  \bigskip

 A knot is a circle smoothly embedded in 3-dimensional Euclidean space or its compactification, the 3-sphere. Two knots are regarded as the same if one can be smoothly deformed into the other.\footnote{ A finite collection of mutually disjoint knots is called a link. Just as in the case of knots, links are regarded only up to smooth deformation. For the sake of simplicity, we will restrict our attention to knots. However, all of the results here apply equally well to links.}

The mathematical theory of knots emerged from the smoky ruins of Lord Kelvin's  ``vortex atom theory," a hopelessly optimistic theory of matter of the nineteenth century in which atoms appeared as microscopic vortices of \ae ther.
Kelvin was inspired by theorems of Hermann von Helmholtz on vortex motion as well as poisonous smoke-ring laboratory demonstrations of a fellow Scot, Peter Guthrie Tait. (See \cite{silver} for a historical account.)  More than anyone else, Tait recognized the mathematical 
profundity of the nascent subject. He was the author of the first publication with the word ``knot" in its title.

As Tait knew, a knot can be represented by a diagram, a regular 4-valent graph in the plane with a {\it tromp l'oeil} device at each vertex indicating how one arc passes over another,  the ``hidden line" device that has been universally adopted today.  Homeomorphisms of the plane might distort the graph, but they do  not change the knot. Going deeper, a theorem of Kurt Reidemeister from 1926 (also proved independently by J.W. Alexander and his student G.B. Briggs one year later) informs us that two diagrams represent the same knot if and only if one can be converted into the other by a finite sequence of local changes, today called {\it Reidemeister moves}  (see, for example, \cite{cf} \cite{liv}). 

\begin{figure}[htb]
\begin{center}
\vspace{-.15in}
\includegraphics[height=3 in]{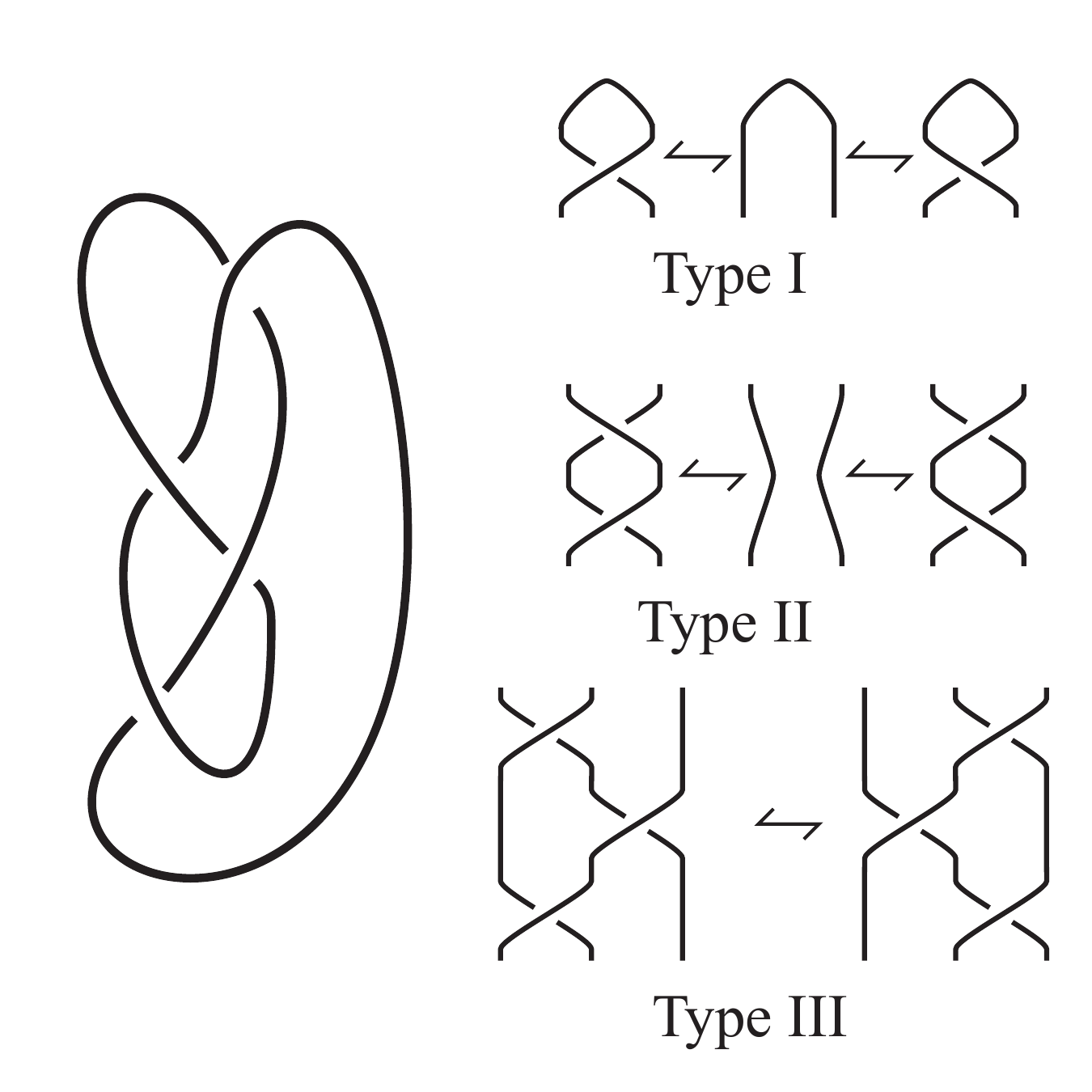}
\caption{Diagram and Reidemeister moves}
\label{Diag.Rmoves}
\end{center}
\end{figure}

Showing that two knots are the same can be relatively easy. However, proving that they are different requires an invariant. A knot invariant is an entity  (number, group, module, etc.) that can be associated to a knot diagram and which is unchanged by any Reidemeister move. 

Reidemeister's theorem converts topological questions about knots into combinatorial problems. Indeed, the first known knot invariants were combinatorial \cite{tait}. As algebraic methods were found, interest in combinatorial invariants waned. However, in the mid 1980's a resurgence of interest followed V.F.R. Jones's discovery and L.H. Kauffman's interpretation of a powerful polynomial knot invariant that could be defined and computed combinatorially \cite{jones}, \cite{kauffman}. Since then, interest in combinatorial knot invariants has remained strong.

Fox $n$-colorings of a knot diagram provide the most elementary but effective combinatorial invariants.
We begin with a brief review of these invariants and their twin siblings,  Dehn $n$-colorings. The first assigns elements of $\Z/n\Z$ (called {\it colors}) to the 1-dimensional arcs of the diagram; the second assigns them to the 2-dimensional regions. In either case, the rules of assignment are determined by the crossings of the diagram. 

Fox $n$-colorings are quite well known, and excellent expositions abound. Dehn $n$-colorings are less well known.\footnote{See pages 185--187 of \cite{kauf}. According to J. Przytycki \cite{prz}, the connection between the two coloring schemes had also occurred some years ago to the late F. Jaeger.}  The next section is intended as a review of the two coloring approaches and the equivalence between them. 

Section 3 describes a third approach in which one colors the 0-dimensional {\sl crossings} of the diagram, and the rules are determined by the regions of the diagram. We obtained it by reformulating ideas of the 1926/27 paper of J.W. Alexander and G.B. Briggs \cite{ab}. For this reason, we refer to the colorings as {\it Alexander-Briggs colorings}. Establishing the relationship between Alexander-Briggs colorings and Fox or Dehn colorings is the goal of the section. 

Colorings organize information in ways that have stimulated new ideas in knot theory. A few such ideas are sketched in the last section.

\section{Fox and Dehn colorings} \label{colorings} Let $D$ be a diagram for a knot $k$, and $n$ any modulus. An {\it arc-coloring} is an assignment of 
{\it colors} $0, 1, \ldots, n-1$ (regarded mod $n$) to the arcs of $D$. 
An arc-coloring is a {\it Fox $n$-coloring} if at every crossing, twice the color of the over-crossing arc is equal to the sum of the colors of the under-crossing arcs, as in Figure \ref{Fox_rule}. 

\begin{figure}%[thb]
\begin{center}
\vspace{-.25in}
\includegraphics[height=1.5 in]{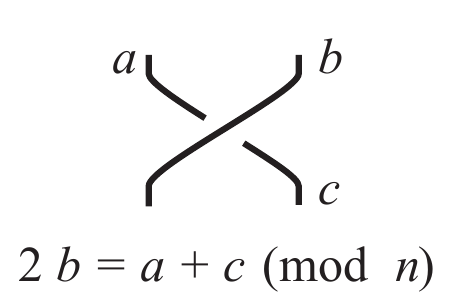}
\caption{Fox $n$-coloring rule}
\label{Fox_rule}
\end{center}
\end{figure}

An {\it $r$-crossing knot} is a knot that has a diagram with $r$ but no fewer crossings. 
Figure \ref{Fox_colorings} gives an example of a 5-coloring of a 4-crossing knot sometimes referred to as Listing's knot \footnote{Johann Benedict Listing (1808-1882) was a student of Gauss and a pioneer in the study of topology. In fact, he is responsible for the name of the subject. Tait learned of Listing's investigation of knots from his life-long friend, the physicist James Clerk Maxwell.} It appears in tables of knots as $4_1$. The figure also shows a 7-coloring of the $9_{42}$.

\begin{figure}
\begin{center}
\vspace{-.15in}
\includegraphics[height=3 in]{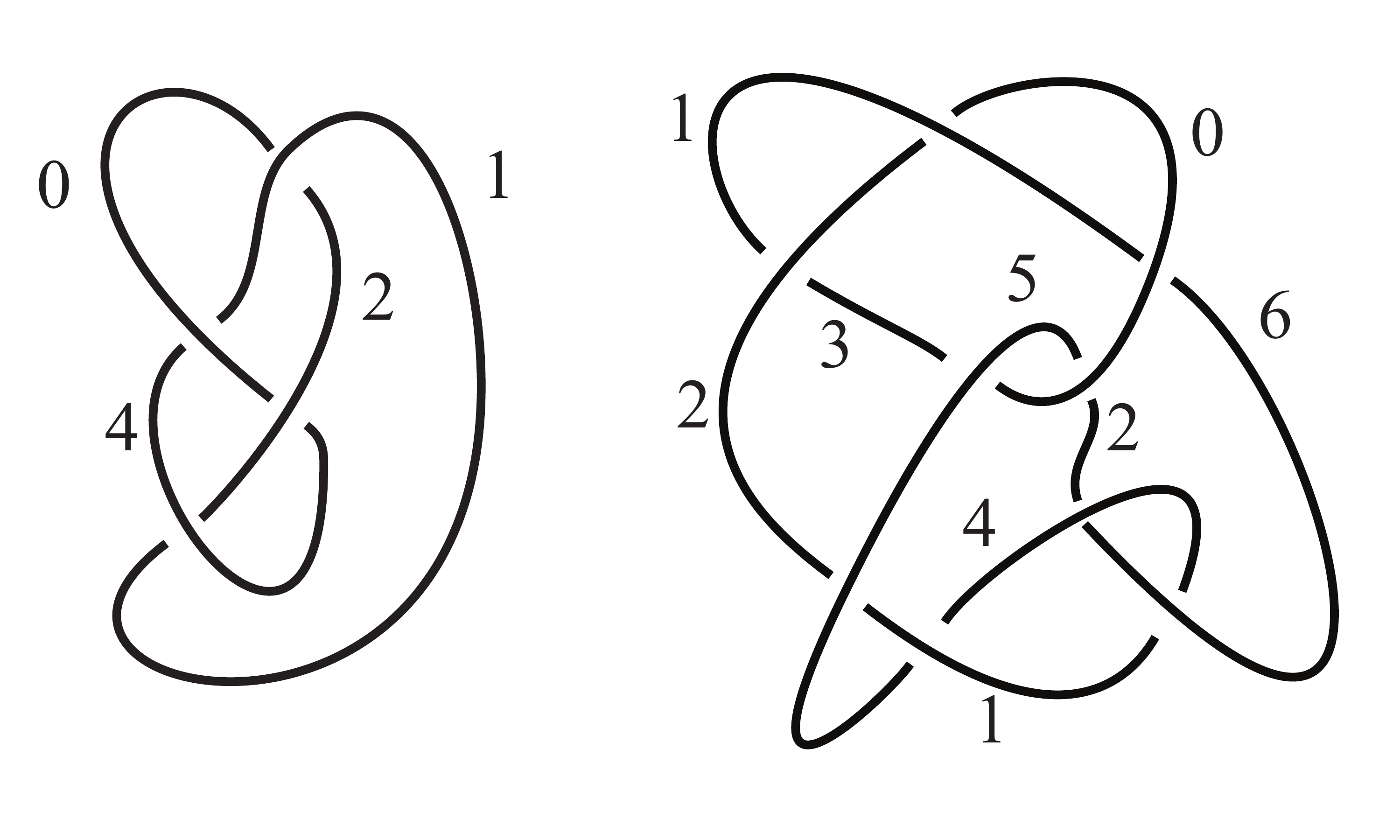}
\caption{Fox $n$-colorings}
\label{Fox_colorings}
\end{center}
\end{figure}

An elementary argument using Reidemeister moves shows that the number of Fox $n$-colorings does not depend on the specific diagram for $k$ that we use. Hence it is an invariant of $k$. Moreover, the linearity of the coloring condition at a crossing implies that the arc-wise sum of two Fox $n$-colorings  is again a Fox $n$-coloring. With a bit more work, one sees that the set of Fox $n$-colorings forms a module over the ring $\Z/n\Z$. The module is also an invariant of $k$.

Every diagram admits $n$  monochromatic Fox $n$-colorings, assigning the same color to each arc. Such arc-colorings are said to be {\it trivial}, and they comprise a submodule. We consider Fox $n$-colorings modulo trivial Fox $n$-colorings. Elements of the quotient module are said to be {\it based}, and they are uniquely represented by Fox $n$-colorings in which an arbitrary but fixed arc, called a {\it basing arc},  is colored by $0$. 

The earliest mention of such invariants appeared in an exercise of the textbook of R.H. Crowell and R.H. Fox \cite{cf} (see pages 92--93). Fox, who was interested in algebraic invariants, recognized that based Fox $n$-colorings are in one-to-one correspondence with homomorphisms from the fundamental group $\pi=\pi_1(\S^3 \setminus k)$ to the dihedral group $D_{2n}= \< \a, \t \mid \t^2, \a^n, (\t \a)^2\>$. 
The correspondence relies on the Wirtinger presentation of $\pi$: 
\begin{equation}\label{wirtinger} \pi=\<x_0, x_1, \ldots, x_m \mid r_1, \ldots, r_m\>. \end{equation}
Here $x_0, \ldots, x_m$ correspond to the arcs of the diagram, having oriented each component. (The colorings will be independent of the orientation.) At each crossing we have a relation of the form 
$x_i x_j = x_k x_i$, where $x_i$ corresponds to the over-crossing arc while $x_j$  is the under-crossing arc on the left as we 
travel above in the preferred direction and $x_k$ is the arc on the right. Any one relation is a consequence of the remaining relations, and hence one relation is omitted from the presentation (\ref{wirtinger}), as in Figure \ref{Wirt}. We remind the reader that $\pi$ is the free group on the generators modulo the smallest normal subgroup containing the set of relators $x_ix_jx_i^{-1}x_k^{-1}$. 

\begin{figure}
\begin{center}
\vspace{-.25in}
\includegraphics[height=2 in]{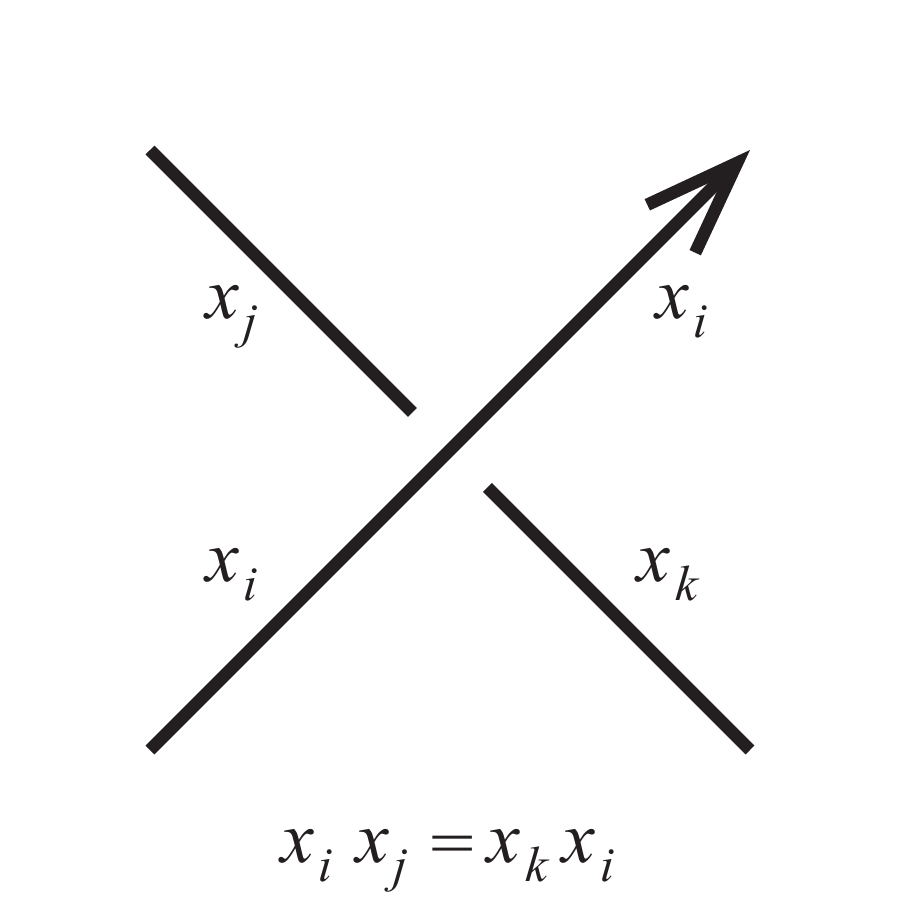}
\caption{Wirtinger relation}
\label{Wirt}
\end{center}
\end{figure}

One easily verifies that given a nontrivial based Fox $n$-coloring of the diagram, the mapping 
that sends each generator $x_i$ to the reflection $\a^{c_i} \t \in D_{2n}$, where $c_i$ is the color assigned to the $i$th arc, determines a nontrivial homomorphism from $\pi$ to $D_{2n}$. Conversely, any nontrivial homomorphism arises in this manner. 

Wirtinger presentations are the most commonly used knot group presentations today. However, about the year 1910, Max Dehn introduced another presentation that has advantages for both combinatorial and geometric group theory. Generators of the Dehn presentation correspond to the regions of a knot diagram, with some region $R^*$ being associated with the identity element. For notational convenience, we use the same letter to denote a region and its associated generator. Again, relations correspond to crossings. If $R_i, R_j, R_k, R_l$ are the regions at a crossing, as in Figure \ref{Dehn}, then the associated relation is $R_i R_j^{-1} R_k R_l^{-1}$. (See \cite{ls} for additional information.)

By translating Fox's observations about  the  relationship between $n$-colorings and dihedral group 
epimorphisms, we can discover the idea of a {\it Dehn $n$-coloring}. It is a labeling of the regions of the diagram with elements of $\Z/n\Z$ such that the unbounded region is labeled 0, and if $a, b, c, d$ are assigned respectively to the regions $R_i, R_j, R_k, R_l$ of Figure \ref{Dehn}, we have $a + b = c + d$. Figure \ref{Dehn_colorings} gives an example of a Dehn 5-coloring of Listing's knot as well as a 
7-coloring of $9_{42}$.

\begin{figure}
\begin{center}
\vspace{-0.15in}
\includegraphics[height=2.25 in]{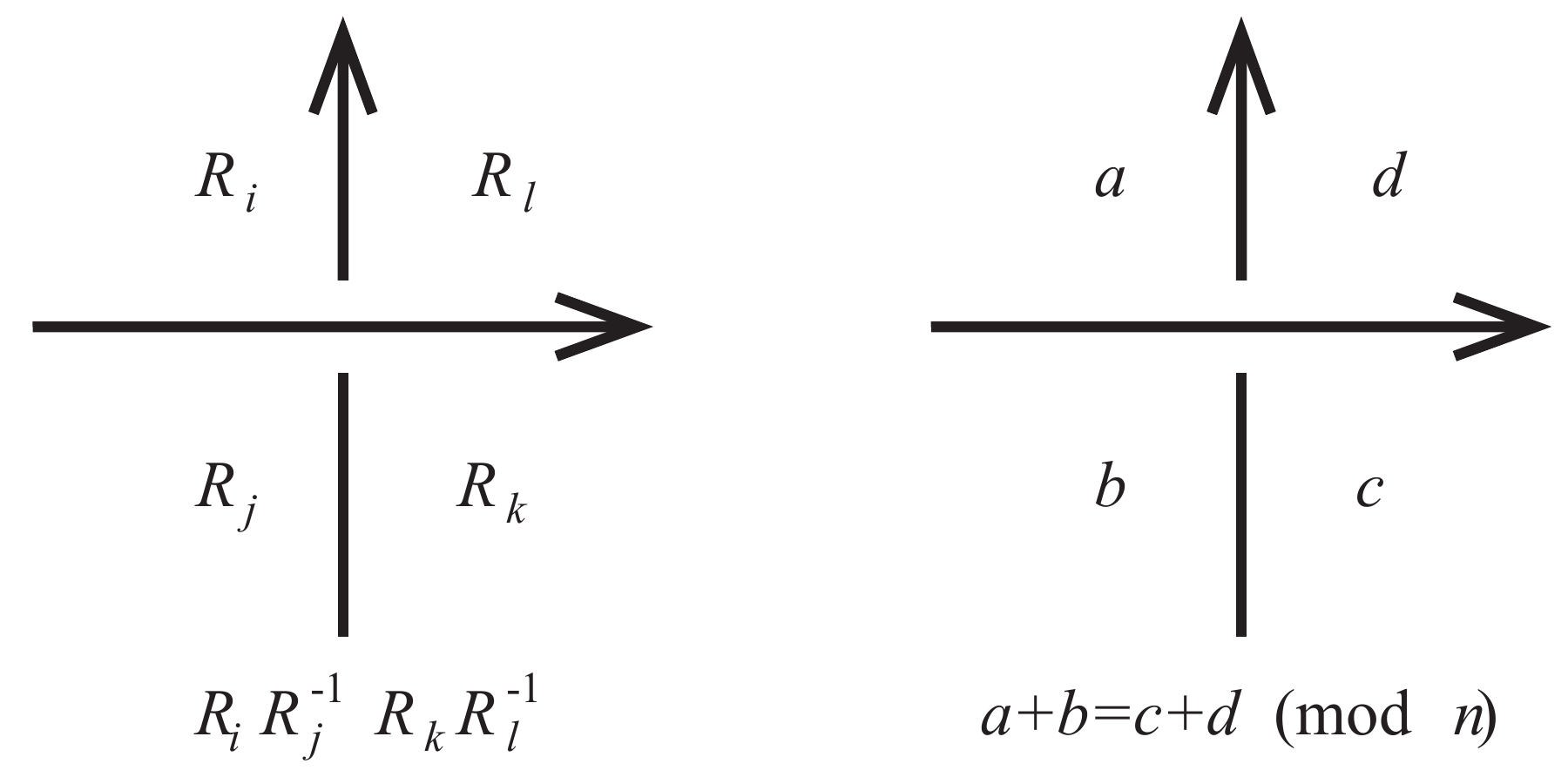}
\caption{Dehn relation and Dehn $n$-coloring rule}
\label{Dehn}
\end{center}
\end{figure}

\begin{figure}
\begin{center}
\vspace{-.15in}
\includegraphics[height=3 in]{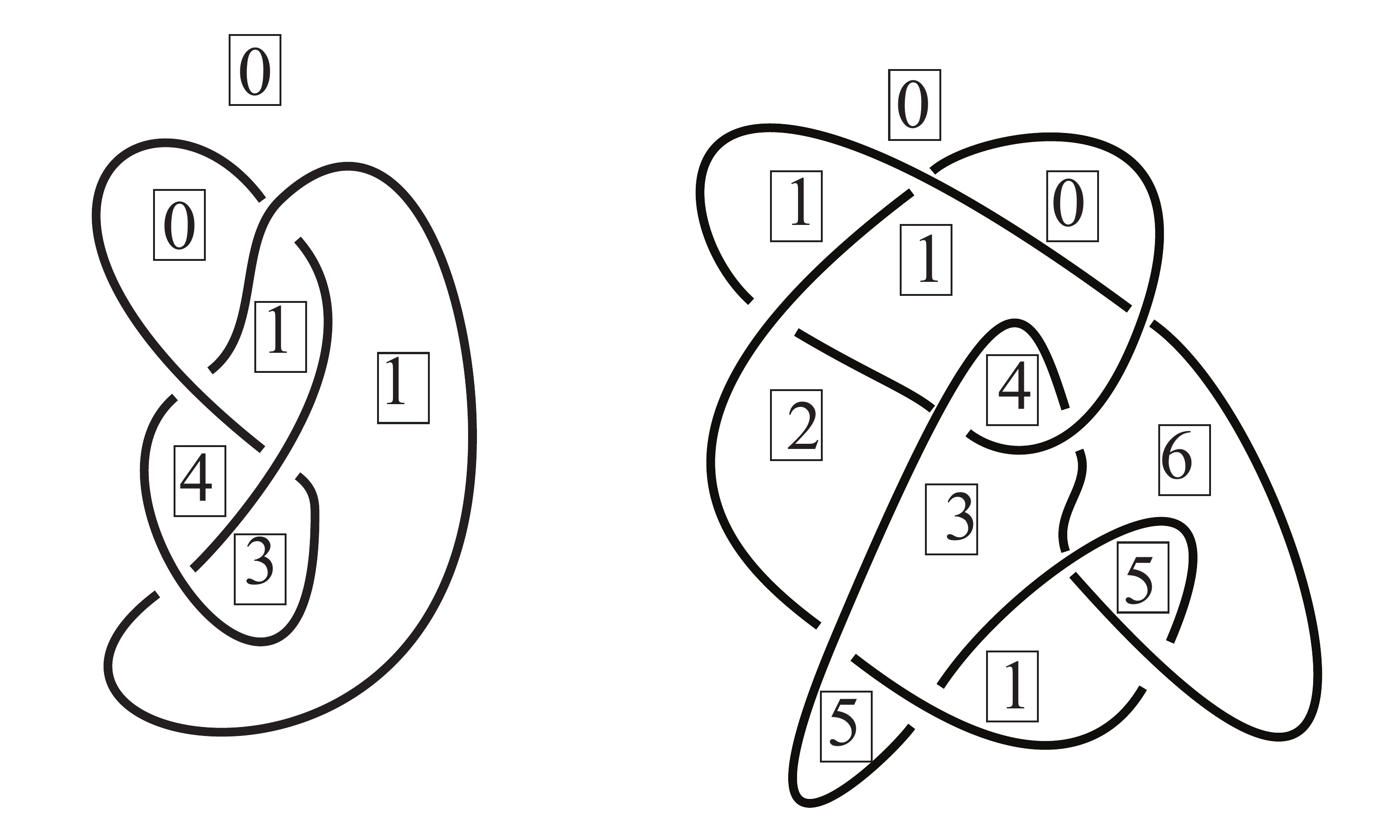}
\caption{Dehn $n$-colorings}
\label{Dehn_colorings}
\end{center}
\end{figure}

Given a Dehn $n$-coloring, it is easy to obtain a Fox $n$-coloring: assign to each arc the sum of the colors of two bounding regions separated by that arc. 

The process of obtaining a Dehn $n$-coloring from a Fox $n$-coloring is more interesting. Given a knot diagram $D$, consider any arc-coloring with elements of $\Z/n\Z$. (Such a coloring need not be a Fox $n$-coloring.) Choose an oriented path $\gamma$ from any region $R$ to any another $R'$ that is {\it generic} in the sense that it crosses arcs transversally and avoids crossings. Assume that $R$ has been colored with an element $a$ of $\Z/n\Z$. If $\gamma$ crosses an arc labeled $b$, then assign to the entered region the color $b-a$. Continue inductively until reaching $R'$. We call the value assigned to $R'$ the result of {\it integration} along $\gamma$. The arc-coloring is {\it conservative} if integration is independent of the path from $R$ to $R'$.

\begin{lemma} \label{conservative} An arc-coloring of $D$ is conservative if and only if it is a Fox $n$-coloring. \end{lemma}

\begin{proof}  Assume that the arc-coloring is a Fox $n$-coloring. It sufficient (and necessary) to prove that integration along any closed path $\gamma$ returns the the initial color of the region. We use induction on the number $N$ of crossings enclosed by $\gamma$. If $N$ is zero, then the result is obvious. Consider a small circle $\delta$ around a crossing that is enclosed by $\gamma$. The claim is easily checked for such a path. Adding $\gamma$ and $\delta$ along the boundary of a thin ribbon (Figure 3) results in another closed path $\gamma'$. Integration along $\gamma$ and $\gamma'$ yield identical results. However, $\gamma'$ encloses $N-1$ crossings. 

Conversely, if the arc-coloring is not a Fox $n$-coloring, then integration along a closed path about some crossing of the diagram will not return the initial color. Hence the arc-coloring is not conservative. \end{proof}

Lemma \ref{conservative} gives a well-defined process for passing from a 
Fox $n$-coloring to a Dehn $n$-coloring. Determine colors for each region by using paths from $R^*$, which is colored trivially. 

We leave it to the reader to check that the process of passing from a Fox $n$-coloring to a Dehn $n$-coloring is the inverse of the process of passing from a Dehn $n$-coloring to a Dehn $n$-coloring. 

Since integration respects the module structures on the sets of Fox or Dehn $n$-colorings, we have shown the following theorem. 

\begin{theorem}\label{FoxDehn} Given any diagram of a knot, integration induces an isomorphism from the module of Fox $n$-colorings to the module of Dehn $n$-colorings. \end{theorem} 

Recall that monochromatic Fox $n$-colorings of a knot diagram assign the same color, say $a$, to every arc. Under integration, such colorings correspond to Dehn $n$-colorings that assign $0$ and $a$ to the diagram in checkerboard fashion. We will call such colorings {\it trivial}. We can consider Dehn $n$-colorings modulo trivial colorings. As in the case of Fox $n$-colorings, elements of the quotient module are said to be {\it based}. They are uniquely represented by Dehn $n$-colorings that assign $0$  to both the unbounded region and an adjacent bounded region.  The arc separating these regions is a basing arc for the associated Fox colorings. 

\begin{cor} Given any diagram of a knot, integration induces an isomorphism from the module of based Fox $n$-colorings to the module of based Dehn $n$-colorings. \end{cor}

%%%%%%%%%%%%%%%%%%%%%%%
\section{Alexander-Briggs colorings} In \cite{ab}, J.W. Alexander and G.B Briggs present a combinatorial method for computing certain homological knot invariants known as torsion numbers. There is a well-known relationship with the Fox (or Dehn) colorings of a knot and such torsion numbers. From it a third type of coloring arises, which we now describe. 

Alexander and Briggs did not use the {\sl trompe l'oeil} effect for depicting one arc crossing over another. Instead they depicted an oriented knot by its generic projection in the plane as a 4-valent graph, marking corners with small dots so that an insect crawling in the positive sense along the upper arc would always have the dotted corners on its right. (We have replaced ``lower arc" and ``on its left" by ``upper arc" and ``on its right" in Alexander's entomological convention \cite{alex} so that the well-known Fox $n$-coloring condition would not have to be altered.)  
Since it is likely that the authors got the pictorial idea from the referenced papers of Tait, we will refer to any such diagram as 
a {\it Tait diagram} of the knot.\footnote{At times, Tait went further by experimenting with two types of markers. He imagined them as silver and copper coins, inspiring a verse by Maxwell  \cite{cg}: 
\verse{\quad But little Jacky Horner\\
Will teach you what is proper,\\
   \quad So pitch him, in his corner,\\
Your silver and your copper.  }}

Let $D$ be Tait diagram, and $n$ any modulus. By a {\it vertex-coloring} we mean an assignment of colors $\{0, 1, \ldots, n-1\}$ (regarded modulo $n$) to the vertices of $D$.  A vertex-coloring is an {\it Alexander-Briggs $n$-coloring} if in each region the sum of of the colors of un-dotted vertices minus the colors of dotted vertices ---the {\it weighted vertex sum}---vanishes.  Figure \ref{AB_Listing} gives an example of an Alexander-Briggs 5-coloring of Listing's knot and illustrates Tait's dot-notation. 

\begin{figure}
\begin{center}
\vspace{-.15in}
\includegraphics[height=3 in]{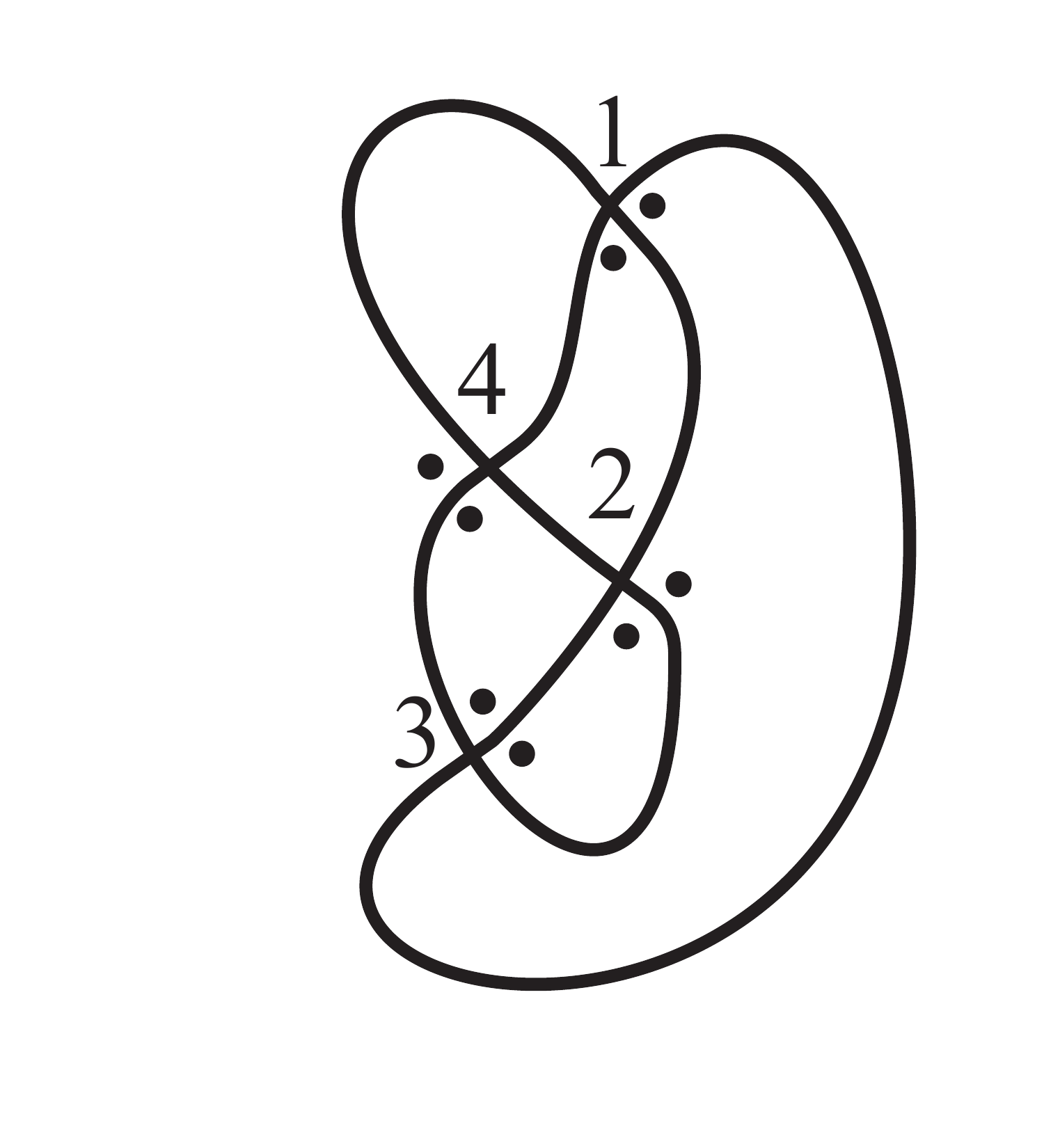}
\caption{Alexander-Briggs n-coloring}
\label{AB_Listing}
\end{center}
\end{figure}

Given a based Dehn $n$-coloring of a knot diagram, one obtains an Alexander-Briggs $n$-coloring. We explain this first for alternating diagrams. 

Consider a Dehn $n$-coloring $D$ of an alternating diagram of a knot $k$. We orient the diagram and let $\bar D$ be the resulting Tait diagram.  At each vertex there are two dotted regions and two undotted regions. Color the vertex with the color of either dotted region minus the color of the diagonally opposite undotted region. The Dehn coloring condition ensures that the result does not depend on which dotted region we use. 

In an alternating diagram we will see two types of regions: regions for which the overcrossing arc at each vertex of $D$ is the one on the left, as viewed from inside the region, and regions for which it is the one on the right.  (These two types of regions alternate in checkerboard fashion.)
 Let $R$ be a region of the first kind.  Let $R_i$, $i\in \Z/n$, be the adjoining regions in clockwise cyclic order, $a_i$ the Dehn color of $R_i$, and $v_i$ the vertex of $R$ between $R_{i-1}$ and $R_i$.  Then $R$ and $R_i$ are both dotted or both undotted at $v_i$, since the dots are on the same side of the overcrossing arc.  Hence the color of $v_i$ will be $a_i-a_{i-1}$ if $v_i$ is dotted, and $a_{i-1}-a_i$ if $v_i$ is undotted.  It follows immediately that the signed vertex sum for $R$ is zero.  For the other type of region, the argument is similar, with the signs all reversed.

If the diagram $D$ is not alternating, then identify a set ${\cal C}$ of crossings with the property that if the sense of each crossing in ${\cal C}$ is changed, then the diagram becomes alternating. (There are exactly two choices for the set ${\cal C}$, and they are complementary sets.)  Given a based Dehn $n$-coloring of $D$, determine colors  for each vertex of $\bar D$ as above, but multiply the color by $-1$ if the crossing is among those in ${\cal C}$.  Since changing a crossing changes whether $R$ shares its dot status with $R_i$ or with $R_{i-1}$, these sign changes are exactly what is needed to satisfy the Alexander-Briggs condition. Figure \ref{AB_9crossing} illustrates. The the crossings in ${\cal C}$ are circled. 

\begin{figure}
\begin{center}
\includegraphics[height=2.8 in]{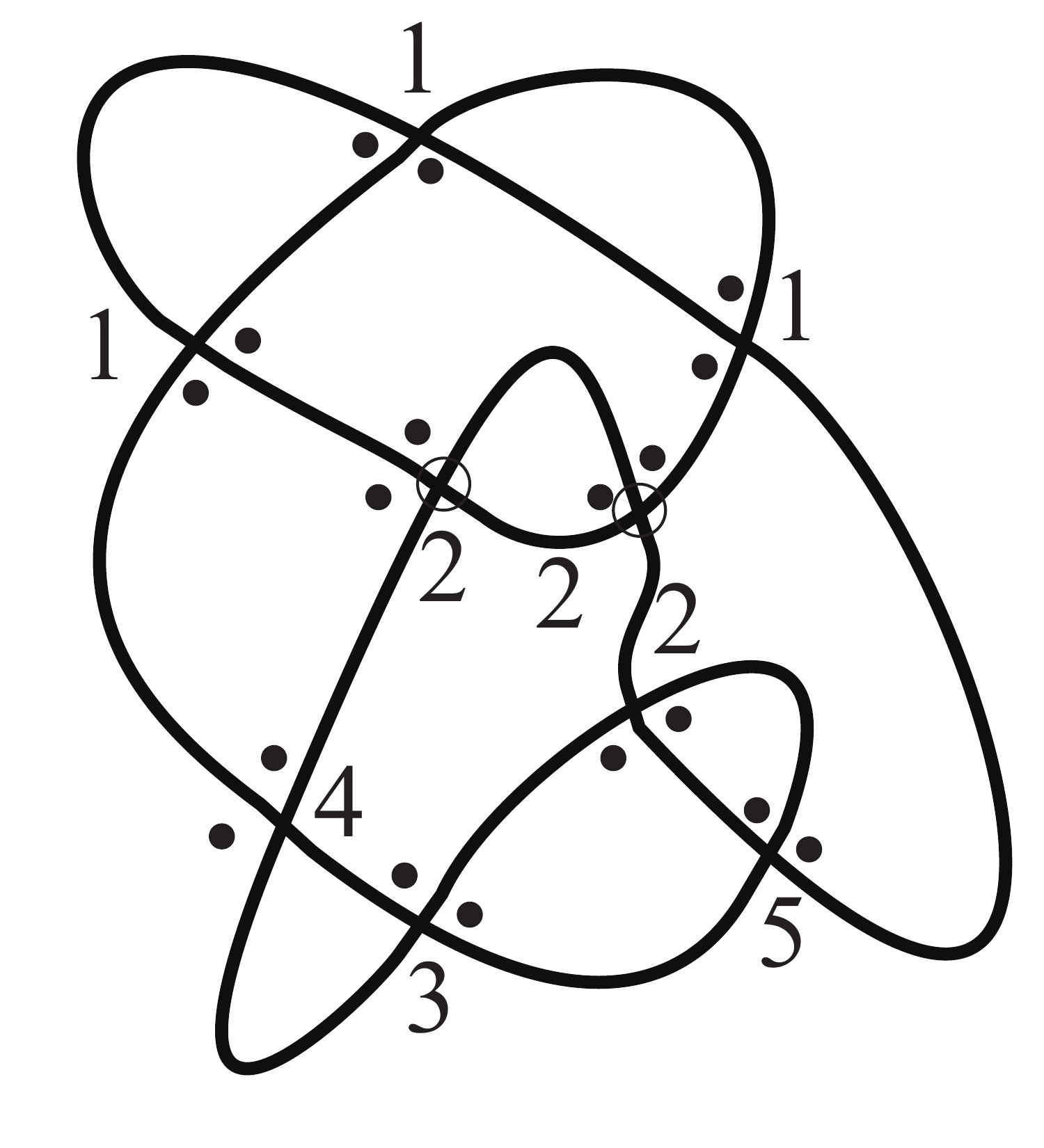}
\caption{Alexander-Briggs coloring of a nonalternating diagram}
\label{AB_9crossing}
\end{center}
\end{figure}

The map from the set of  Dehn $n$-colorings of $D$ to Alexander-Briggs $n$-colorings of $\bar D$ is obviously linear. Moreover, trivial (checkerboard) Dehn $n$-colorings are mapped to the constant-zero Alexander-Briggs $n$-coloring. Hence there is an induced map $\Phi$ from based Dehn $n$-colorings of $D$ to Alexander-Briggs $n$-colorings of $\bar D$. The kernel of $\Phi$ is easily seen to consist of only the constant-zero Dehn $n$-colorings of $D$. Hence $\Phi$ is a bijective correspondence. 

The alert reader will be concerned about the choice of orientation with which we began. If we reverse the orientation, then each of the vertex colors is replaced by its inverse, and hence $\Phi$ becomes $-\Phi$. In the case of a link, reversing the orientation of a component inverts colors at each vertex corresponding to overcrossings of the component.

We have shown: 

\begin{theorem}\label{DehnAlex}  Given any diagram of a knot, $\Phi$ is an isomorphism from the module of based Dehn $n$-colorings to the module of Alexander-Briggs $n$-colorings. \end{theorem}

\section{Taking knot colorings in other directions}

The three approaches to knot coloring that we have described are only part of the tale. As Fox well understood, knot colorings pack information about the homology groups of 2-fold branched cyclic covers of $\S^3$ branched over the knot. But yet there is more to explore. We mention two relatively recent directions.

Fox $n$-colorings are a special case of quandle colorings. A {\it quandle} is a
set with a binary operation that satisfies three axioms that
correspond to the Reidemeister moves. On the set ${\mathbb Z}/n$, for example, the operation $ a\triangleleft b= 2 b -a \pmod{n}$ defines a quandle.
The idea of using quandle colorings to detect knotting
was introduced in Winker's dissertation \cite{winker}
and subsequently discussed by Kauffman and Harary \cite{hk}. 
The theory is taken further in \cite{CJKLS}.

%Quandle
%cocycle invariants as introduced in \cite{CJKLS} can be used to group
%different colorings into classes that depend upon the cocycle value
%that is taken on the coloring. Such invariants can be used to
%determine geometric properties of the knots, and they generalize to
%create invariants for higher dimensional knots.

For any $n$, the group of colors $\Z/n\Z$ embeds naturally in the {\it additive circle group} $\T = \R/\Z$. Why not extend our palette of colors to the entire circle? This is the main idea of \cite{sw1}, \cite{sw2}. The set of Fox $\T$-colorings of a knot turns out to be a compact abelian group, and conjugation in the knot group by a meridian induces a homeomorphism. Suddenly we enter the world of algebraic dynamics, where new invariants such as periodic point structure and topological entropy arise.

Fox, Dehn and Alexander-Briggs $n$-colorings help to organize knot information in stimulating ways, just as knot diagrams themselves do.
The charm that they hold for both students and researchers makes it likely that they will continue to inspire novel perspectives about knots for years to come.

 %%%%%%%%%%%%%%%%%%%%%%%%%

 \bigskip

\ni Department of Mathematics and Statistics,\\
\ni University of South Alabama\\ Mobile, AL 36688 USA\\
\ni Email: \\
\ni carter@southalabama.edu\\
\ni  silver@southalabama.edu\\
\ni swilliam@southalabama.edu
\end{document}